\documentclass[article,12pt]{article}
\usepackage[lmargin=25mm, tmargin=25mm, textwidth=160mm, textheight= 221mm]{geometry}
\usepackage{tikz,multicol}
\usetikzlibrary{arrows,snakes,backgrounds}
\usepackage{amsthm,amsmath,amssymb}
\usepackage{graphicx}
\usepackage[colorlinks=true,citecolor=black,linkcolor=black,urlcolor=blue]{hyperref}
\usepackage{enumerate}

\theoremstyle{plain}
\newtheorem{theorem}{Theorem}
\newtheorem{lemma}[theorem]{Lemma}
\newtheorem{corollary}[theorem]{Corollary}
\newtheorem{proposition}[theorem]{Proposition}
\newtheorem{fact}[theorem]{Fact}

\theoremstyle{definition}

\theoremstyle{remark}

\newcommand{\tk}{\tilde{K}}
\newcommand{\ed}{{\rm ed}}
\newcommand{\forb}{{\rm Forb}}
\newcommand{\dist}{{\rm dist}}
\newcommand{\vk}{V(K)}

\newcommand{\vwk}{{\rm VW}(K)}
\newcommand{\vbk}{{\rm VB}(K)}
\newcommand{\ewk}{{\rm EW}(K)}
\newcommand{\ebk}{{\rm EB}(K)}
\newcommand{\egk}{{\rm EG}(K)}

\newcommand{\x}{{\bf x}}
\newcommand{\one}{{\bf 1}}
\newcommand{\mkp}{{\bf M}_K(p)}
\newcommand{\K}{\mathcal{K}}
\newcommand{\F}{\mathcal{F}}
\newcommand{\hh}{\mathcal{H}}

\title{\bf On the Edit Distance of Powers of Cycles}

\author{
Zhanar Berikkyzy \qquad  Ryan R. Martin \qquad Chelsea Peck\\
\small Department of Mathematics\\[-0.8ex]
\small Iowa State University\\[-0.8ex]
\small Ames, Iowa, U.S.A\\
\small\tt \{zhanarb,rymartin\}@iastate.edu
}

\date{August 30, 2015}

\begin{document}

\maketitle

\begin{abstract}
  The edit distance between two graphs on the same labeled vertex set is defined to be the size of the symmetric difference of the edge sets. The edit distance function of a hereditary property $\mathcal{H}$ is a function of $p\in [0,1]$ that measures, in the limit, the maximum normalized edit distance between a graph of density $p$ and $\mathcal{H}$.

  In this paper, we address the edit distance function for $\forb(H)$, where $H=C_h^t$, the $t^{\rm th}$ power of the cycle of length $h$. For $h\geq 2t(t+1)+1$ and $h$ not divisible by $t+1$, we determine the function for all values of $p$. For $h\geq 2t(t+1)+1$ and $h$ divisible by $t+1$, the function is obtained for all but small values of $p$. We also obtain some results for smaller values of $h$.

\end{abstract}

\section{Introduction} 

The edit distance in graphs was introduced independently by Axenovich, K\'{e}zdy, and Martin \cite{A-K-M} and by Alon and Stav \cite{A-S}. The question considered is ``Given a class of graphs $\hh$ what is the minimum number $m=m(n)$ such that for every graph on $n$ vertices, there is a set of $m$ edge-additions and edge-deletions that ensure the resultant graph is a member of $\hh$?'' For every hereditary property $\hh$, there is a $p=p(\hh)$ such that the Erd\H{o}s-R\'{e}nyi random graph $G(n,p)$ is asymptotically extremal~\cite{A-S}. A \emph{hereditary property} is a family of graphs that is closed under isomorphism and the taking of induced subgraphs.

The edit distance function of a hereditary property $\hh$ is a function of $p\in [0,1]$ that measures, in the limit, the maximum normalized edit distance between a graph of density $p$ and $\hh$. A \emph{principal hereditary property}, denoted $\forb(H)$, is a hereditary property that consists of the graphs with no induced copy of a single graph $H$. Most of the known edit distance functions are of the form $\forb(H)$. These include the cases where $H$ is a split graph~\cite{M_split} (including cliques and independent sets), complete bipartite graphs $K_{2,t}$~\cite{M-MK} and $K_{3,3}$~\cite{M-T} and cycles $C_h$ where $h$ is small~\cite{M_sym}. In this paper, we compute the edit distance function for powers of cycles.

For positive integers $t$ and $h$, the $t^{\rm th}$ power of a cycle of length $h$ is denoted $C_h^t$ and has vertex set $\{1,\ldots,h\}$, where two vertices are adjacent in $C_h^t$ if and only if their distance is at most $t$ in $C_h$. 

The notation in this paper primarily comes from \cite{M_survey}.
The \emph{edit distance} between graphs $G$ and $G'$ on the same labeled vertex set is denoted $\dist(G,G')$ and satisfies $\dist(G,G')=|E(G)\bigtriangleup E(G')|$. The edit distance between a graph $G$ and a hereditary property $\hh$ is
\begin{equation*}
{\rm dist}(G, \hh)=\min \lbrace \dist(G,G') : V(G)=V(G'), G'\in\hh \rbrace.
\end{equation*}

The edit distance function of a hereditary property $\hh$ measures the maximum distance of a density $p$ graph from $\hh$, i.e.
\begin{equation*}
\ed_{\hh}(p)=\lim_{n \rightarrow \infty} \max \lbrace \dist(G,\hh) : |V(G)|=n, |E(G)|=\lfloor p \tbinom{n}{2} \rfloor \rbrace / \tbinom{n}{2}.
\end{equation*}

Balogh and Martin~\cite{B-M} showed that this limit exists and is equal to $\underset{n\rightarrow\infty}{\lim} \mathbb{E}[\dist(G(n,p),\hh)]$, using the result of Alon and Stav in \cite{A-S}. The function has a number of interesting properties:
\begin{proposition}[Balogh-Martin~\cite{B-M}] \label{prop:basic}
   If $\hh$ is a hereditary property, then $\ed_{\hh}(p)$ is continuous and concave down over $p\in [0,1]$.
\end{proposition}

By the proposition above, the function $\ed_{\hh}$ achieves its maximum in $[0,1]$. We denote this maximum value by $d_{\hh}^*$, and the the set of all values of $p$ for which the maximum is achieved by $p_{\hh}^*$.

A \emph{colored regularity graph (CRG)}, $K$, is a complete graph with a partition of the vertices into white $\vwk$ and black $\vbk$, and a partition of the edges into white $\ewk$, gray $\egk$, and black $\ebk$. We say that a graph $H$ \emph{embeds in} $K$, denoted $H \mapsto K$, if there is a function $\varphi : V(H)\rightarrow\vk$ so that if $h_1 h_2 \in E(H)$, then either $\varphi(h_1)=\varphi(h_2)\in \vbk$ or $\varphi(h_1)\varphi(h_2)\in \ebk\cup \egk$, and if $h_1 h_2 \notin E(H)$, then either $\varphi(h_1)=\varphi(h_2)\in \vwk$ or $\varphi(h_1)\varphi(h_2)\in \ewk\cup \egk$.

Given a hereditary property $\hh$, it is easy to see that it can be expressed as $\hh=\bigcap\lbrace\forb(H): H\in \F(\hh)\rbrace$ for some family of graphs $\F(\hh)$. We denote $\K(\hh)$ to be the subset of CRGs such that no forbidden graph embeds into them, i.e. $\K(\hh)=\lbrace K : H \not\mapsto K, \forall H\in \F(\hh)\rbrace$. In our case, $\K(\hh)=\lbrace K : H \not\mapsto K\rbrace$ for $\hh=\forb(H)$ and so a CRG $K$ is \emph{a sub-CRG} of $\tk$ if $K$ can be obtained by deleting vertices of $\tk$.

For every CRG $K$ we associate a function $g$ on $[0,1]$ defined by

\begin{equation}
\label{quad_prog}
g_K(p)=\min \lbrace \x^T\mkp\x : \x^T\one=1, \x\geq \textbf{0} \rbrace,
\end{equation}
where

\begin{equation}
   [\mkp]_{ij}=\left\{\begin{array}{ll}
                         p, & \mbox{if $v_iv_j\in\ewk$ or $v_i=v_j\in\vwk$;} \\
                         1-p, & \mbox{if $v_iv_j\in\ebk$ or $v_i=v_j\in\vbk$;} \\
                         0, & \mbox{if $v_iv_j\in\egk$.}
                      \end{array}\right.
\end{equation}

The $g$ function of CRGs can be used to compute the edit distance function. Balogh and Martin~\cite{B-M} proved that $\ed_{\hh}(p)=\underset{K\in\K(\hh)}{\inf} g_K(p)$ and Marchant and Thomason~\cite{M-T} further proved that the infimum is achieved by some $K$, i.e. $\ed_{\hh}(p)=\underset{K\in\K(\hh)}{\min} g_K(p)$. So, for every $p\in[0,1]$, there is a CRG $K\in\K(\hh)$ such that $\ed_{\hh}(p)=g_K(p)$. It is also shown in~\cite{M-T} that in order to find such CRG we only need to look at so called $p$-core CRGs. A CRG $\tk$ is \emph{$p$-core} if $g_{\tk}(p)<g_{K}(p)$ for every sub-CRG $K$ of $\tk$.

The CRG with $r$ white vertices, $s$ black vertices and all edges gray is denoted $K(r,s)$. The \emph{clique spectrum} of the hereditary property $\hh=\forb(H)$, denoted $\Gamma(\hh)$, is the set of all pairs $(r,s)$ such that $H \not\mapsto K(r,s)$. It is easy to see that, for any hereditary property $\hh$ its clique spectrum $\Gamma=\Gamma(\hh)$ can be expressed as a Ferrers diagram. That is, if $r\geq 1$ and $(r,s)\in\Gamma$, then $(r-1,s)\in\Gamma$ and if $s\geq 1$ and $(r,s)\in\Gamma$, then $(r,s-1)\in\Gamma$.  An \emph{extreme point} of a clique spectrum $\Gamma$ is a pair $(r,s)\in \Gamma$ such that $(r+1,s)$ and $(r,s+1)$ do not belong to $\Gamma$. The set of all extreme points of $\Gamma$ is denoted by $\Gamma^*$.

Define the function $\gamma_{\hh}(p)=\min \lbrace g_{K(r,s)}(p) : (r,s)\in \Gamma(\hh) \rbrace$. Clearly, $\ed_{\hh}(p)\leq \gamma_{\hh}(p)$. Moreover, one only need consider the extreme points rather than whole $\Gamma$ itself, that is, $\gamma_{\hh}(p)=\min \lbrace g_{K(r,s)}(p) : (r,s)\in \Gamma^*(\hh) \rbrace$.

In this paper, the hereditary properties we consider are of the form $\hh=\forb(C_h^t)$. Since Theorem 20 in~\cite{M_survey} gives $\ed_{\forb(K_h)}(p)=p/(h-1)$, we will assume that $h\geq 2t+2$. For convenience, we denote $\ell_a=\left\lceil\frac{h}{t+a+1}\right\rceil$, for $a\in\lbrace 0,\ldots,t\rbrace$. We also denote $p_0={\ell_t}^{-1}$. The motivation for these values will be discussed in Section \ref{sec:proof_main}.

The main results of this paper are Theorems \ref{Thm:established_gamma} and \ref{Thm:main}.

\begin{theorem} \label{Thm:established_gamma}
   Let $t\geq 1$ and  $h\geq\max\{t(t+1),4\}$ be integers.  For all $p\in [0,1]$,
   \begin{align*}
      \gamma_{\hh}(p) &= \min_{a\in\{0,1,\ldots,t\}} \left\{\frac{p(1-p)}{a(1-p)+\left(\ell_a-1\right)p} \right\}, & \mbox{ if $(t+1)\mid h$;} \\
      \gamma_{\hh}(p) &= \min_{a\in\{0,1,\ldots,t\}} \left\{\frac{p}{t+1},\frac{p(1-p)}{a(1-p)+\left(\ell_a-1\right)p}\right\}, & \mbox{ if $(t+1)\mid\!\!\!\not\;\; h$.}
   \end{align*}
\end{theorem}

\noindent \textbf{Note:} If $a=0$, then $\frac{p(1-p)}{a(1-p)+\left( \ell_a -1 \right) p}=\frac{p(1-p)}{\left( \ell_0 -1 \right) p}$, which we define to be $\frac{1-p}{ \ell_0 -1 }$ at $p=0$.

\begin{theorem} \label{Thm:main}
   Let $t\geq 1$ and $h\geq 2t(t+1)+1$ be positive integers and let $\hh=\forb(C_h^t)$.

   If $(t+1)\mid\!\!\!\not\;\; h$, then for $0\leq p\leq 1$, and if $(t+1)\vert h$ then for $p_0\leq p\leq 1$, we have that
   \begin{equation}
      \ed_{\hh}(p)=\gamma_{\hh}(p).
   \end{equation}
\end{theorem}

\begin{corollary}\label{cor:cycles}
   Let $h\geq 5$ be a positive integer and $\hh=\forb(C_h)$.

   \begin{itemize}
      \item If $h$ is even, then for ${\lceil h/3\rceil}^{-1}\leq p\leq 1$,
      \begin{equation*}
         \ed_{\hh}(p)=\min \left\lbrace  \frac{p(1-p)}{1-p+(\lceil h/3\rceil-1)p}, \frac{1-p}{\lceil h/2\rceil-1}\right\rbrace .
      \end{equation*}

      \item If $h$ is odd, then for $0\leq p\leq 1$,
      \begin{equation*}
         \ed_{\hh}(p)=\min \left\lbrace \frac{p}{2}, \frac{p(1-p)}{1-p+(\lceil h/3\rceil-1)p}, \frac{1-p}{\lceil h/2\rceil-1}\right\rbrace .
      \end{equation*}
   \end{itemize}
\end{corollary}

It was shown in \cite{M_sym} and in \cite{M-T}, respectively, that
\begin{align*}
   \ed_{\forb(C_3)}(p)=p/2 \qquad\mbox{and}\qquad \ed_{\forb(C_4)}(p)=p(1-p) .
\end{align*}

\noindent It follows from this and the above corollary that when $t=1$, the furthest graph from $\forb(C_h)$ is a graph which has density $p^*=1/(\lceil h/2 \rceil - \lceil h/3 \rceil +1)$ when $h\geq 4$ and $h\not\in \lbrace 4,7,8,10,16\rbrace$, and has density $p^*=1/(1+\sqrt{\lceil h/3 \rceil -1})$ when $h\in \lbrace 4,7,8,10,16 \rbrace$. Also, observe that the maximum value of the edit distance function can be an irrational number.

Our proof techniques often require us to compare the $g$ function of a CRG to one of the individual functions that are given in Theorem~\ref{Thm:established_gamma}. However, when $h$ is large enough at most 3 of these functions are necessary to define $\gamma_{\hh}$.

\begin{corollary}\label{cor:established_gamma}
   Let $t\geq 2$ and $h\geq 4t^2+10t+24$ be positive integers. Recall $\ell_t=\left\lceil\frac{h}{2t+1}\right\rceil$ and $p_0=\ell_t^{-1}$. Then
   \begin{itemize}
      \item If $(t+1)\mid h$, then for $p_0\leq p\leq 1$,
      \begin{equation*}
         \ed_{\hh}(p)=\min\left\{\frac{p(1-p)}{t(1-p)+(\ell_t-1)p}, \frac{1-p}{\ell_0-1}\right\} .
      \end{equation*}

      \item If $(t+1)\mid\!\!\!\not\;\; h$, then for $0\leq p\leq 1$,
      \begin{equation*}
         \ed_{\hh}(p)=\min\left\{\frac{p}{t+1}, \frac{p(1-p)}{t(1-p)+(\ell_t-1)p}, \frac{1-p}{\ell_0-1}\right\} .
      \end{equation*}
   \end{itemize}
\end{corollary}

The rest of the paper is organized as follows: Section~\ref{sec:defns} provides some definitions and basic results, Section~\ref{sec:comp_gamma} gives the proof of Theorem~\ref{Thm:established_gamma}, Section~\ref{sec:forbidden_cycles} gives Lemma~\ref{lem:cycle_length_short} which is the key lemma for the proof of Theorem~\ref{Thm:main}, Section~\ref{sec:proof_main} gives the proof of Theorem~\ref{Thm:main}, Section~\ref{sec:proofs} gives the proofs of some helpful lemmas and facts, and Section~\ref{sec:divisibility_p_small} gives some concluding remarks.

\section{Definitions and Tools}
\label{sec:defns}

All graphs considered in this paper are simple. For standard graph theory notation please see \cite{West}, for the edit distance notation please see \cite{M_survey}. A sub-CRG $K'$ of a CRG $K$ is a \emph{component} if it is maximal with respect to the property that, for all $v,w\in V(K')$, there exists a path consisting of white and black edges entirely within $K'$. It is easy to compute the $g$ function of a CRG given the $g$ function of its components:
\begin{proposition}[\cite{M_sym}]
\label{Prop:components}
Let $K$ be a CRG with components $K^{(1)},\ldots, K^{(r)}$ and $p\in [0,1]$. Then $\left( g_K(p)\right) ^{-1}=\sum_{i=1}^{r}\left( g_{K^{(i)}}(p)\right)^{-1}$.
\end{proposition}

\noindent Note that by Proposition~\ref{Prop:components}, $g_{K(r,s)}(p)=\left( \frac{r}{p}+\frac{s}{1-p} \right)^{-1}$.

Let $K$ be a CRG, $v\in\vk$, and let $\x$ be an optimal solution to the quadratic program~(\ref{quad_prog}). The \emph{weight} of $v$, denoted $\x(v)$, is the entry corresponding to $v$ of the vector $\x$. We say that $w\in\vk$ is a \emph{gray neighbor} of $v\in\vk$ if $w$ is adjacent to $v$ via a gray edge. White and black neighbors are defined analogously. The set of all gray neighbors of $v$ is denoted by $N_G(v)$ and the number of vertices adjacent to $v$ via gray edges is denoted by $\deg_G(v)$, i.e. $\deg_G(v)=|N_G(v)|$.

In contrast, the \emph{gray degree} of $v$, denoted $d_G(v)$, is the sum of the weights of gray neighbors of $v$, i.e. $d_G(v)=\sum \lbrace \x(w): w\in N_G(v) \rbrace$. Similarly, the \emph{white degree} of $v$, denoted $d_W(v)$, is the sum of the weights of the white neighbors of $v$ plus the weight of $v$ if and only if it is a white vertex. The \emph{black degree} of $v$, denoted $d_B(v)$, is the sum of the weights of the black neighbors of $v$ plus the weight of $v$ if and only if it is a black vertex. So, $d_G(v)+d_W(v)+d_B(v)=1$ for all $v\in\vk$.

The number of common gray neighbors of vertices $v$ and $w$ is denoted by $\deg_G(v,w)$. The \emph{gray codegree} of vertices $v$ and $w$, denoted $d_G(v,w)$, is the sum of the weights of the common gray neighbors of $v$ and $w$.  For a set of vertices $\lbrace v_1,v_2,\ldots,v_{\ell}\rbrace$, we say $v_1v_2\cdots v_{\ell}$ is a \emph{gray path} if $v_iv_{i+1}\in \egk$ for $i=1,\ldots, \ell-1$. Analogously, we say $v_1v_2\cdots v_{\ell} v_1$ is a \emph{gray cycle} if $v_1v_{\ell}\in \egk$ and $v_iv_{i+1}\in \egk$ for $i=1,\ldots, \ell-1$. Proposition~\ref{Prop:pcore} gives a structural classification of $p$-core CRGs.

\begin{proposition}[Marchant-Thomason~\cite{M-T}]\label{Prop:pcore}
Let $K$ be a $p$-core CRG.
\begin{itemize}
\item If $p=1/2$, then all of the edges of $K$ are gray.
\item If $p<1/2$, then $\ebk=\emptyset$ and there are no white edges incident to white vertices.
\item If $p>1/2$, then $\ewk=\emptyset$ and there are no black edges incident to black vertices.
\end{itemize}
\end{proposition}

Proposition~\ref{Prop:M_sym} gives a formula for $d_G(v)$ for all $v\in\vk$ and Proposition~\ref{Prop:x_bd} uses this to give a bound on the weight of each $v$.

\begin{proposition}[\cite{M_sym}] \label{Prop:M_sym}
Let $p\in(0,1)$ and $K$ be a $p$-core CRG with optimal weight function $\x$.
\begin{itemize}
\item If $p\leq 1/2$, then $\x(v)=g_K(p)/(1-p)$ for all $v\in\vwk$, and
\begin{equation*}
d_G(v)=\frac{p-g_K(p)}{p}+\frac{1-2p}{p} \x(v), \bigskip \text{ for all } v\in \vbk.
\end{equation*}
\item If $p\geq 1/2$, then $\x(v)=g_K(p)/p$ for all $v\in\vbk$, and
\begin{equation*}
d_G(v)=\frac{1-p-g_K(p)}{1-p}+\frac{2p-1}{1-p} \x(v), \bigskip \text{ for all } v\in \vwk.
\end{equation*}
\end{itemize}
\end{proposition}

\begin{proposition}[\cite{M_sym}]\label{Prop:x_bd}
Let $p\in(0,1)$ and $K$ be a $p$-core CRG with optimal weight function~$\x$.
\begin{itemize}
\item If $p\leq 1/2$, then $\x(v)\leq g_K(p)/(1-p)$ for all $v\in\vbk$.
\item If $p\geq 1/2$, then $\x(v)\leq g_K(p)/p$ for all $v\in\vwk$.
\end{itemize}
\end{proposition}

\section{Proof of Theorem~\ref{Thm:established_gamma}: Computation of the $\gamma_{\hh}$ function}
\label{sec:comp_gamma}

In this section we compute the $\gamma_{\hh}$ function, which gives an upper bound for the edit distance function. Recall that for any $t\geq 1$, $h\geq 2t+2$ and $a\in\{0,\ldots,t\}$, we denote $\ell_a = \left \lceil \frac{h}{t+a+1}\right\rceil$.

\begin{proof}[Proof of Theorem~\ref{Thm:established_gamma}]
Our first observation is the value of the chromatic number of $C_h^t$, denoted $\chi(C_h^t)$.
\begin{proposition}[Prowse-Woodall~\cite{P-W}] \label{prop:chromatic}
Let $t\geq 1$ and $h\geq\max\{t+1,3\}$ be positive integers. Let $h=q(t+1)+r$, where $r\in\{0,\ldots,t\}$. Then, $\chi(C_h^t)=t+\lceil r/q\rceil+1$. In particular, if $h\geq\max\{t(t+1),3\}$, then
\begin{align*}
   \chi(C_h^t)=\left\lbrace \begin{array}{ll}
                               t+1, & \mbox{ if $(t+1) \mid h$;} \\
                               t+2, & \mbox{ if $(t+1)\mid\!\!\!\not\;\; h$.}
                            \end{array} \right.
\end{align*}
\end{proposition}

Let $h\geq\max\{t(t+1),2t+2\}$ and $\chi=\chi(C_h^t)$. Denote the vertices of $C_h^t$ by $\{1,\ldots,h\}$ such that distinct $i$ and $j$ are adjacent if and only if $|i-j|\leq t\pmod{h}$. For each $a\in\{0,\ldots,t\}$, we first show that $\left(a,\ell_a-1\right)\in\Gamma=\Gamma(\forb(C_h^t))$ and then show that $\left(a,\ell_a\right)\not\in\Gamma$. We will also show that if $\chi>t+1$ then $\{(t+1,0),\ldots,(\chi-1,0)\}\subset\Gamma$ but that $(t+1,1)\not\in\Gamma$.

This will imply that $\Gamma^*\subseteq\left\{(a,\ell_a-1) : a=0,1,\ldots,t \right\}\cup\{(\chi-1,0)\}$, which is a stronger result than we need.~\\

\noindent\textbf{Case 1: $a\in\{0,\ldots,t\}$.}~\\
\indent First, we show that $\left(a,\ell_a-1\right)\in\Gamma$. By contradiction, assume there is a partition of $V(C_h^t)$ into $a$ independent sets and $\ell_a-1$ cliques. Let $k=\ell_a-1$, and let $C_1,\ldots,C_k$ be the cliques. We may assume that the vertices in each $C_i$ are consecutive. This is because if $j_1$ and $j_2$ are in the same clique, then by the nature of adjacency in the power of a cycle, every vertex between $j_1$ and $j_2$ is adjacent to every member of the clique, and hence can be added to the clique. Thus, $|C_i|\leq t+1$ for $i=1,\ldots,k$.

For $i=1,\ldots,k-1$, let $B_i$ be the set of vertices between $C_i$ and $C_{i+1}$, and let $B_k$ be the set of vertices between $C_k$ and $C_1$. The sets $B_i$ might or might not be empty. If some $|B_i|\geq a+1$, then the first $a+1\leq t+1$ vertices form a clique and so must be in different independent sets, which is not possible since there are only $a$ independent sets.  Therefore, $|B_i|\leq a$ for $i=1,\ldots,k$.

Consequently, we need $k(t+a+1)\geq h$ in order to cover $C_h^t$ with $a$ independent sets and $k$ cliques. Hence, $k\geq\ell_a$, a contradiction to our choice of $k$. Thus $\left(a,\ell_a-1\right)\in\Gamma$ for $a=0,\ldots,t$.

Next, we show that $\left(a,\ell_a\right)\not\in\Gamma$. Again, let $k=\ell_a-1$. For $i=1,\ldots,k$, let $S_i=\{(i-1)(t+a+1)+1,\ldots,i(t+a+1)\}$ and let $S_{k+1}=\{1,\ldots,h\}-\cup_{i=1}^k S_i$.  For $i=1,\ldots,k$, let $C_i$ be the first $t+1$ vertices of $S_i$ and let $C_{k+1}$ be the first $\min\{t+1,|S_{k+1}|\}$ vertices of $S_{k+1}$. For $j=1,\ldots,a$, let $A_j$ consist of the $(t+1+j)^{\rm th}$ vertex of $S_1,\ldots,S_k$ and the $(t+1+j)^{\rm th}$ vertex of $S_{k+1}$ if $|S_{k+1}|\geq t+1+j$.

The sets $\left(A_1,\ldots,A_a,C_1,\ldots,C_{k+1}\right)$ form a partition of $V(C_h^t)$. Clearly each $C_i$, $i=1,\ldots,k$, is a clique of size $t+1$ and since there is a clique of size $t+1$ between pairs of vertices in each $A_j$, each $A_j$ is an independent set. Thus $\left(a,\ell_a\right)\not\in\Gamma$ for $a=0,\ldots,t$.~\\

\noindent\textbf{Case 2: $a\geq t+1$.}~\\
\indent If $(t+1)\mid h$, then Proposition~\ref{prop:chromatic} gives that $C_h^t$ can be partitioned into $t+1$ independent sets and so $(t+1,0)\not\in\Gamma$. If $(t+1)\mid\!\!\!\not\;\; h$, then Proposition~\ref{prop:chromatic} gives that $\chi\geq t+2$ and since $C_h^t$ cannot be partitioned into fewer than $\chi$ independent sets, we have $(t+1,0),\ldots,(\chi-1,0)\in\Gamma$. Since $C_h^t$ can be partitioned into $\chi$ independent sets, $(\chi,0)\not\in\Gamma$.

Finally, let $k=\lceil h/(t+1)\rceil-1$. For $j=1,\ldots,t+1$, let $A_j=\{(i-1)(t+1)+j : i=1,\ldots,k\}$. Let $C_0=\{k(t+1)+1,\ldots,h\}$. The sets $\left(A_1,\ldots,A_{t+1},C_0\right)$ form a partition of $V(C_h^t)$. Clearly, $C_0$ is a clique of size at most $t+1$ and since there are at least $t$ vertices between pairs of vertices in each $A_j$, each $A_j$ is an independent set. Thus $\left(t+1,1\right)\not\in\Gamma$.

Using Proposition~\ref{Prop:components}, if $h=q(t+1)+r$ where $r\in\{0,\ldots,t\}$, then
\begin{align*}
   \gamma_{\hh}(p) &= \underset {a \in \lbrace 0,1,\ldots,t\rbrace }{\min} \left\lbrace \frac{p(1-p)}{a(1-p)+\left(\ell_a-1\right)p} \right\rbrace, & \mbox{ if $r=0$;} \\
   \gamma_{\hh}(p) &= \underset {a \in \lbrace 0,1,\ldots ,t \rbrace }{\min} \left\lbrace \frac{p}{t+\lceil r/q\rceil}, \frac{p(1-p)}{a(1-p)+\left(\ell_a-1\right)p}\right\rbrace, & \mbox{ if $r\neq 0$.}
\end{align*}

Restricting ourselves to $h\geq\min\{t(t+1),4\}$, we have the result in the statement of the theorem.
\end{proof}

\section{Forbidden Cycles}
\label{sec:forbidden_cycles}

Before we can prove Theorem~\ref{Thm:main}, we need to study the properties of the CRGs into which $C_h^t$ does not embed. Recall that we may assume $h\geq 2t+2$. An important property of such CRGs is that the set of lengths of gray cycles on black vertices is restricted, as is shown in Lemma~\ref{lem:cycle_length_short}. Its proof needs the technical inequalities in Facts~\ref{fact:two_part_simple} and~\ref{fact:bounds_on_h}. For completeness, we give their proofs in Section \ref{sec:proofs}.

\begin{fact}
\label{fact:two_part_simple}
Let $h, x, y$ be positive integers. Then
\begin{enumerate}[(a)]
\item $\left \lfloor h/x \right \rfloor \geq y$ if and only if $\left \lfloor h/y \right \rfloor \geq x$.\label{case:xygeq}
\item $\left \lceil h/x \right \rceil \leq y$ if and only if $\left \lceil h/y \right \rceil \leq x$.\label{case:xyleq}
\end{enumerate}
\end{fact}

\begin{fact}\label{fact:bounds_on_h} Let $t\geq 1$, $h\geq\max\{t(t-1),2t+2\}$, and $a\in\{0,\ldots,t-1\}$ be positive integers. Then $\left\lceil\frac{h}{t+a+1}\right\rceil\leq\left\lfloor\frac{h}{t}\right\rfloor$. \end{fact}

Lemma~\ref{lem:cycle_length_short} is a key lemma in proving our main result of Theorem~\ref{Thm:main}.

\begin{lemma} \label{lem:cycle_length_short}
Let $p\in (0,1/2]$ and let $t\geq 1$ and $h\geq 2t+2$ be integers. Let $\tk$ be a $p$-core CRG with exactly $a$ white vertices such that $C_h^t \not\mapsto \tk$. Let $K$ be the sub-CRG of $\tk$ induced by the set of all black vertices of $\tk$. Then, the following occurs:
\begin{enumerate}[(a)]
   \item If $a\in\{0,\ldots,t-1\}$ and $h\geq t^2-t$, then $K$ has no gray cycle which has length in $\left\{\left\lceil\frac{h}{t+a+1}\right\rceil, \ldots, \left\lfloor\frac{h}{t}\right\rfloor\right\}$.\label{case:a}
   \item If $a=t$, then $|\vk|\leq\ell_t-1$.\label{case:t}
   \item If $a\geq t+1$, then $(t+1)\mid\!\!\!\not\;\; h$ and $V(K)=\emptyset$.\label{case:tplusone}
\end{enumerate}
\end{lemma}

\noindent\textbf{Note:} We interpret a gray cycle of length 2 to be a gray edge.

\begin{proof}[Proof of Lemma~\ref{lem:cycle_length_short}]
Denote the vertices of $C_h^t$ by $\{1,\ldots,h\}$ such that distinct $i$ and $j$ are adjacent if and only if $|i-j|\leq t\pmod{h}$.~\\

\noindent\textbf{Partition:} Let $K$ have a gray cycle on vertex set $\{v_1,\ldots,v_k\}$ such that $v_iv_{i+1}$ is a gray edge, where the indices are taken modulo $k$. We describe a partition of $V(C_h^t)$, which gives an interval of forbidden gray cycle lengths. We will construct at most $a$ independent sets and $k$ cliques $C_1,\ldots,C_k$ such that there is no edge between nonconsecutive cliques.

Partition $V(C_h^t)$ into $k$ sets of consecutive vertices $S_1,\ldots,S_k$, with each set $S_i$ of size either $\lceil h/k\rceil$ or $\lfloor h/k\rfloor$. We will eventually construct the at most $a$ independent sets and $k$ cliques $C_1,\ldots,C_k$ with $C_i\subseteq S_i$ such that there is no edge between $C_i$ and $C_{i'}$ unless $|i-i'|=1\pmod{k}$. 

If $a=0$, then simply let $C_i=S_i$ for $i=1,\ldots,k$. Using Fact~\ref{fact:numth}, each $C_i$ has size at least $t$ and so nonconsecutive sets have no edge between them. Fact~\ref{fact:numth} is a simple observation of number theory.

\begin{fact}\label{fact:numth} A set of size $h$ can be partitioned into sets of size $t$ or $t+1$ if and only if $h\geq t(t-1)$. Moreover, for any $k\in\left\{\lceil h/(t+1)\rceil,\ldots,\lfloor h/t\rfloor\right\}$, such a partition exists with exactly $k$ parts. \end{fact}

So, we assume $a\geq 1$ and choose $a'\in\lbrace\lfloor h/k\rfloor-t,\lceil h/k\rceil-(t+1)\rbrace$ such that $0\leq a'\leq a$. This is possible as long as both (a) $0\leq\lfloor h/k\rfloor-t$ and (b) $\lceil h/k\rceil-(t+1)\leq a$. (This is only nontrivial if $k\mid h$, in which case at least one of the two choices of $a'$ will be in $0,\ldots,a$.)

If $a'=0$, again let $C_i=S_i$ for $i=1,\ldots,k$. If $a'\geq 1$, let $A_j$ consist of the $j^{\rm th}$ vertex of each of $S_1,\ldots,S_k$ and let $C_i=S_i-\cup_{j=1}^{a'}A_j$. Observe that if $a'\geq 1$, then $|S_i|\geq t+1$ and so there are at least $t$ vertices between each pair of vertices in every $A_j$. Therefore, $A_j$ is an independent set for $j=1,\ldots,a'$. We have $|C_i|\leq t+1$ so $C_i$ is a clique for $i=1,\ldots,k$. In addition, $|C_i|\geq t$ and so there are no edges between $C_i$ and $C_{i'}$ unless $|i-i'|\pmod k$.

The mapping, for all $a\geq 0$, is as follows: Map each $A_j$ to a different white vertex and $C_i$ to $v_i$ for $i=1,\ldots,k$. If $a=0$, Fact~\ref{fact:numth} gives that $K$ has no cycle with length in $\left\{\lceil h/(t+1)\rceil,\ldots,\lfloor h/t\rfloor\right\}$. If $a\geq 1$, Fact~\ref{fact:two_part_simple}, gives that $K$ has no cycle with length in
\begin{align}
   \left\lbrace\left\lceil\frac{h}{t+a+1}\right\rceil, \ldots, \left\lfloor\frac{h}{t}\right\rfloor\right\rbrace , \label{eq:part1}
\end{align}
and (\ref{eq:part1}) is valid in the case of $a=0$ also.~\\

\noindent\textbf{Case \textit{(\ref{case:a})}.}~\\
\indent The result is given by (\ref{eq:part1}). It suffices to show that $\left\lceil\frac{h}{t+a+1}\right\rceil\leq\left\lfloor\frac{h}{t}\right\rfloor$. Fact~\ref{fact:bounds_on_h} gives that this holds if $h\geq t^2-t$.~\\

\noindent\textbf{Case \textit{(\ref{case:t})}.}~\\
\indent In this case, we use a second partition. Partition $V(C_h^t)$ into $k+1$ consecutive parts, $S_1,\ldots,S_{k+1}$, where $k=\lceil h/(2t+1)\rceil-1$ and $r=h-(k-1)(2t+1)$. Since $h\geq 2t+2$, $k\geq 1$. Let $|S_1|=\cdots=|S_{k-1}|=2t+1$, $|S_k|=\lceil r/2\rceil$ and $|S_{k+1}|=\lfloor r/2\rfloor$. Note that $t+1\leq |S_{k+1}|\leq |S_k|\leq 2t+1$.

For $j=1,\ldots,t$, let $A_j$ consist of the $j^{\rm th}$ vertex in each part and let $C_i=S_i-\bigcup_{j=1}^tA_j$.

Each of $A_1,\ldots,A_t$ is an independent set. Furthermore, there are no edges between $C_i$ and $C_{i'}$ if $i\neq i'$.  Therefore, $K$ has at most $k=\lceil h/(2t+1)\rceil-1$ vertices; otherwise, $A_1,\ldots,A_t$ can be mapped arbitrarily to each of the $t$ white vertices and $C_1,\ldots,C_{k+1}$ can be mapped arbitrarily to $k+1$ different black vertices in $K$.~\\

\noindent\textbf{Case \textit{(\ref{case:tplusone})}.}~\\
\indent If $(t+1)\mid h$, then $\chi(C_h^t)=t+1$ and $\tk$ having at least $t+1$ white vertices means that $C_h^t$ embeds in $\tk$, a contradiction. If $(t+1)\mid\!\!\!\not\;\; h$, then partition $V(C_h^t)$ into $k=\lfloor h/(t+1)\rfloor+1$ parts $S_1,\ldots,S_k$ of consecutive vertices, each of $S_1,\ldots,S_{k-1}$ of size $t+1$. For $j=1,\ldots,t+1$, let $A_j$ consist of the $j^{\rm th}$ vertex in each $S_i$ for $i=1,\ldots,k-1$. The graph induced by $V(C_h^t)-\bigcup_{j=1}^{t+1}A_j$ forms a clique of size at most $t$ in $S_k$. Since all vertices in $K$ are black, this clique will embed into any vertex of $V(K)$. Thus $V(K)=\emptyset$.
\end{proof}

\section{Proof of Theorem~\ref{Thm:main}: $\ed_{\hh}=\gamma_{\hh}$}
\label{sec:proof_main}

We will use Lemma~\ref{lem:cycle_length_short} to prove Theorem~\ref{Thm:main}. Recall that $h\geq 2t(t+1)+1\geq t(t+1)$. By Proposition~\ref{prop:chromatic}, this means $\chi(C_h^t)=t+1$ if $(t+1)\mid h$ and $\chi(C_h^t)=t+2$ if $(t+1)\mid\!\!\!\not\;\; h$.

\begin{proof}[Proof of Theorem~\ref{Thm:main}]
By definition $\ed_{\hh}(p)\leq \gamma_{\hh}(p)$ for all $p \in [0,1]$, so we need to show equality. ~\\

\noindent \textbf{Case 1:} $p\in [1/2,1]$.~\\
\indent Fact~\ref{fact:tlarge} below establishes that $\gamma_{\hh}(p)=\frac{1-p}{\ell_0-1}$ for $p\in [1/2,1]$. 

\begin{fact}\label{fact:tlarge}
   Let $h$ and $t$ be positive integers. If $h\geq (t+1)^2+1$, then
   \begin{align*}
      \frac{1-p}{\ell_0-1}\leq\frac{p}{t+1} .
   \end{align*}
   For $a\in\{1,\ldots,t\}$ if $h\geq (t+1)(t+a)+1$, then for all $p\in [1/2,1]$,
   \begin{align*}
      \frac{1-p}{\ell_0-1}\leq\frac{p(1-p)}{a(1-p)+(\ell_a-1)p} .
   \end{align*}
\end{fact}

\noindent\textbf{Note:} The condition $h\geq 2t(t+1)+1$ suffices to achieve all of the conclusions in Fact~\ref{fact:tlarge}.~\\

By Proposition~\ref{prop:values} below, $\ed_{\hh}(p)=\gamma_{\hh}(p)$ for the two values of $p\in\{1/2,1\}$.

\begin{proposition}[Balogh-Martin~\cite{B-M}] \label{prop:values}
   If $\hh$ is a hereditary property, then $\ed_{\hh}(1/2)=\gamma_{\hh}(1/2)$. Moreover, if $K_{\ell}\in\hh$ for all positive integers $\ell$, then $\ed_{\hh}(1)=\gamma_{\hh}(1)=0$ and if $\overline{K_{\ell}}\in\hh$ for all positive integers $\ell$, then $\ed_{\hh}(0)=\gamma_{\hh}(0)=0$.
\end{proposition}

We have $\ed_{\hh}(p)\leq\gamma_{\hh}(p)$ and the two functions are equal at $p=1/2$ and at $p=1$. The function $\gamma_{\hh}(p)$ is linear over $p\in [1/2,1]$ for $h\geq 2t(t+1)+1$. By Proposition~\ref{prop:basic}, $\ed_{\hh}(p)$ is continuous and concave down, so we may conclude that $\ed_{\hh}(p)=\gamma_{\hh}(p)=\frac{1-p}{\ell_0-1}$ for $p\in [1/2,1]$. This concludes Case 1.~\\

Note that Proposition~\ref{prop:values} gives $\ed_{\hh}(0)=\gamma_{\hh}(0)=0$. Let $p\in (0,1/2)$ and $\ed_{\hh}(p)=g_{\tk}(p)$ for some $p$-core CRG $\tk$. Assume by contradiction that $g_{\tk}(p)<\gamma_{\hh}(p)$. Suppose $\tk$ has $a$ white vertices. Recall that for any $t\geq 1$, $h\geq 2t+2$ and $a\in\{0,\ldots,t\}$, we denote $\ell_a = \left \lceil \frac{h}{t+a+1}\right\rceil$. We consider several cases and show that we arrive at a contradiction in each case.~\\

\noindent \textbf{Case 2:} $a\geq t$ and $p\in (0,1/2)$.~\\
\indent If $a\geq t+1$, then by Lemma~\ref{lem:cycle_length_short}\textit{(\ref{case:tplusone})}, $V(K)=\emptyset$. As long as $h\geq\max\{t(t+1),3\}$, Proposition~\ref{prop:chromatic} gives that $\chi(C_h^t)\leq t+2$ with equality only if $(t+1)\mid\!\!\!\not\;\; h$ .  Thus, $a=t+1$ and Proposition~\ref{Prop:components} gives that $g_{\tk}(p)= p/(t+1)$, a contradiction to the assumption that $g_{\tk}(p)<\gamma_{\tk}(p)$.

If $a=t$, then Case~\textit{(\ref{case:t})} of Lemma~\ref{lem:cycle_length_short} gives that $|\vk|\leq \ell_t-1$. Consequently, $g_K(p)\geq\frac{1-p}{\ell_t-1}$. We can partition $\tk$ into $t+1$ sub-CRGs, $K$ and $t$ white vertices, and use Proposition~\ref{Prop:components} to conclude that
\begin{align*}
   \left(g_{\tk}(p)\right)^{-1} &\leq tp^{-1} + \left(\frac{1-p}{\ell_t-1}\right)^{-1} \\
   g_{\tk}(p) &\geq \frac{p(1-p)}{t(1-p)+(\ell_t-1)p}
\end{align*}

Hence, $\ed_{\hh}(p)\geq\gamma_{\hh}(p)$, again a contradiction. This concludes Case 2.~\\

\noindent \textbf{Case 3:} $a\leq t-2$ and $p\in (0,1/2)$.~\\
\indent Recall that $\tk$ is a CRG with $a$ white vertices, with $0\leq a\leq t-2$. By Proposition~\ref{Prop:components}, $g_{\tk}(p)^{-1}=a p^{-1}+g_K^{-1}(p)$. Therefore,
\begin{align}
   g_K(p) < \left(\max_{a'\in\{0,1,\ldots,t\}}\left\{\frac{a'-a}{p}+\frac{\ell_{a'}-1}{1-p}\right\}\right)^{-1} =: g_0(a,t;p) . \label{eq:rel_K_and_tK}
\end{align}

Given our assumptions on $g_K(p)$, Lemma~\ref{lem:degrees} gives lower bounds on the gray degree of vertices and the codegree of pairs of vertices. Recall that $\deg_G(v)$ denotes the \textit{number} of gray neighbors of $v\in V(K)$.
\begin{lemma} \label{lem:degrees}
   Let $p\in (0,1/2)$, $t\geq 1$ be an integer and $a\in\{0,\ldots,t-1\}$. Let $p_0=\ell_t^{-1}=\left\lceil\frac{h}{2t+1}\right\rceil^{-1}$. Let $K$ be a $p$-core CRG with all black vertices such that $g_K(p)<g_0(a,t;p)$. Then
   \begin{enumerate}[(a)]
      \item for every $v\in\vk$, we have $\deg_G(v)\geq \ell_{a+1}$, and \label{case:deg}
      \item for every $v,w\in\vk$, \label{case:codeg}
      \begin{align*}
            \deg_G(v,w) \geq \left\{\begin{array}{ll}
                                       \ell_{a+2}, & \mbox{if $a\leq t-2$;} \\
                                       1, & \mbox{if $a=t-1$ and $p\geq p_0$.}
                                    \end{array}\right.
      \end{align*}
   \end{enumerate}
\end{lemma}
\noindent\textbf{Note:} Since $h\geq 2t+2$, it is the case that $\ell_{a+1}\geq 2$ for $a\leq t-1$ and $\ell_{a+2}\geq 2$ for $a\leq t-2$.

Now we consider the derived graph $F$ with vertex set $\vk$ and edge set $\egk$. Using Lemma~\ref{lem:degrees}, the lower bound on the number of common gray neighbors of $v$ and $w$ gives a structural restriction on this graph. Note that the length of a path is defined to be the number of vertices in said path.
\begin{lemma} \label{lem:cycle_length_long}
   Fix integers $t\geq 1$, $h\geq\max\{t(t-1),2t+2\}$ and $a\in\{0,\ldots,t-1\}$. Recall that $\ell_a=\lceil h/(t+a+1)\rceil$ and let $L=\lfloor h/t\rfloor$.

   Let $F$ be a graph with no cycle with length in $\{\ell_a,\ldots,L\}$ and every pair of vertices either has at least $\ell_{a+2}\geq 2$ common neighbors if $a\leq t-2$ or has at least $1$ common neighbor if $a=t-1$.

   Then $F$ has no cycle of length more than $\ell_a-1$.
\end{lemma}

Now we consider a maximum-length path in the graph $F$. If such a path can be made into a cycle, then Proposition~\ref{prop:Hamiltonian} gives that $F$ must be Hamiltonian. By Lemma~\ref{lem:cycle_length_long}, this means that $|\vk|\leq\ell_a-1$ and, as such, $g_K(p)\geq \frac{1-p}{\ell_a-1}$, which is the $g$ function for the CRG on $\ell_a-1$ black vertices with all edges gray. This is a contradiction to our assumption in (\ref{eq:rel_K_and_tK}) by setting $a'=a$. Proposition~\ref{prop:Hamiltonian} is a common argument in proofs of Hamiltonian cycle results, including the classical theorems of Dirac~\cite{Dirac} and Ore~\cite{Ore}.

\begin{proposition} \label{prop:Hamiltonian}
   Let $F$ be a connected graph. If some path of maximum length forms a cycle, then $F$ is Hamiltonian.
\end{proposition}

So we may assume that every maximum-length path in $F$ is not a cycle. Let $v_1\cdots v_{\ell}$ be such a maximum length path. The common neighbors of $v_1$ and $v_{\ell}$ in $F$ must be on this path, otherwise $F$ has a longer path. From Lemma~\ref{lem:degrees}, it follows that $v_1$ and $v_{\ell}$ have at least $\ell_{a+2}\geq 2$ common neighbors on this path. However, Lemma~\ref{Lem:gap} gives that there can only be one such neighbor, a contradiction.

\begin{lemma} \label{Lem:gap}
   Fix integers $t\geq 1$, $h\geq 2t+2$ and $a\in\{0,\ldots,t-1\}$. Recall that $\ell_a=\lceil h/(t+a+1)\rceil$.  Let $F$ be a graph with no cycle of length longer than $\ell_a-1$, with every vertex having degree at least $\ell_{a+1}\geq 2$ and with every pair of vertices having at least one common neighbor. Furthermore, let $F$ have the property that no maximum length path forms a cycle.

   Let $v_1\cdots v_{\ell}$ be a path of maximum length in $F$. Then $v_1$ and $v_{\ell}$ have exactly one common neighbor $v_c$ on this path. Furthermore, $N(v_1)\subseteq\{v_2,\ldots,v_c\}$ and $N(v_{\ell})\subseteq\{v_c,\ldots,v_{\ell}\}$.
\end{lemma}

This concludes Case 3.~\\

\noindent \textbf{Case 4:} $a=t-1$ and $p\in [p_0,1/2)$.~\\
\indent Recall that $\tk$ is a CRG with $a=t-1$ white vertices. By Proposition~\ref{Prop:components}, $g_{\tk}^{-1}(p)=(t-1) p^{-1}+g_K^{-1}(p)$. Therefore,
\begin{align*}
   g_K(p) < g_0(t-1,t;p) = \left(\max_{a'\in\{0,1,\ldots,t\}}\left\{\frac{a'-(t-1)}{p}+\frac{\ell_{a'}-1}{1-p}\right\}\right)^{-1}  \leq \frac{1-p}{\ell_{t-1}-1} .
\end{align*}

Again, we consider the graph $F$ with vertex set $\vk$ and edge set $\egk$. By Lemma~\ref{lem:degrees}, every vertex in $F$ has degree at least $\ell_t$ and every pair of vertices has at least one common neighbor.  By Lemma~\ref{lem:cycle_length_long}, $F$ has no cycle of length more than $\ell_{t-1}-1$. If there is a maximum-length path that is a cycle, then Proposition~\ref{prop:Hamiltonian} gives that $F$ is Hamiltonian, which means $|\vk|\leq\ell_{t-1}-1$. As a result, $g_K(p)\geq \frac{1-p}{\ell_{t-1}-1}$, a contradiction.

So we may assume that every maximum-length path in $F$ is not a cycle. Let $v_1\ldots v_{\ell}$ be such a maximum-length path such that, in $K$, the sum $\x(v_1)+\x(v_{\ell})$ is the largest among such paths. Let $v_c$ be the unique common neighbor of $v_1$ and $v_{\ell}$.

Let $v_1$ have $d$ neighbors in $F$. Since $v_1$ cannot have neighbors outside of this path, the sum of the weights, in $K$, of the neighbors of $v_1$ satisfy $d_G(v_1)\leq\x(v_2)+\cdots+\x(x_c)$. Notice that if $v_i\in\{v_1,\ldots,v_{c-1}\}$ is a predecessor of a neighbor of $v_1$, then it is an endpoint of a path containing the same $\ell$ vertices, namely $v_iv_{i-1}\cdots v_1v_{i+1}v_{i+2}\cdots v_c\cdots v_{\ell}$. Hence all $d$ predecessors of gray neighbors of $v_1$ (including $v_1$ itself) have weight at most $\x(v_1)$. All other vertices have weight at most $\frac{g_K(p)}{1-p}$. Proposition~\ref{Prop:M_sym} gives
\begin{align*}
   \frac{p-g_K(p)}{p}+\frac{1-p}{p}\x(v_1)=\x(v_1)+d_G(v_1)\leq \x(v_1)+\cdots+\x(v_c)\leq d\x(v_1)+(c-d)\frac{g_K(p)}{1-p} .
\end{align*}

Rearranging the terms, we obtain
\begin{align*}
   g_K(p)\left(\frac{c-d}{1-p}+\frac{1}{p}\right) \geq 1-\x(v_1)\left(d-\frac{1-p}{p}\right) .
\end{align*}

Since $p^{-1}\leq p_0^{-1}=\ell_t$ and $\ell_t<d+1$, we may, by Lemma~\ref{lem:degrees}, lower bound the right-hand side by using $\x(v_1)\leq\frac{g_K(p)}{1-p}$ from Proposition~\ref{Prop:x_bd},
\begin{align*}
   g_K(p)\left(\frac{c-d}{1-p}+\frac{1}{p}\right) &\geq 1-\frac{g_K(p)}{1-p}\left(d-\frac{1-p}{p}\right) \\
   g_K(p)\left(\frac{c}{1-p}\right) &\geq 1 .
\end{align*}

Lemma~\ref{lem:cycle_length_long} bounds the size of the longest cycle, so $c\leq \ell_{t-1}-1$. Thus, $g_K(p)\geq\frac{1-p}{c}\geq\frac{1-p}{\ell_{t-1}-1}\geq g_0(t-1,t;p)$, a contradiction. This concludes Case 4.~\\
 
\noindent\textbf{Case 5:} $a= t-1$ and $p \in (0,p_0)$.~\\
\indent It remains to prove the theorem for $0< p < p_0=\ell_t^{-1}$ in the case where $(t+1)\mid\!\!\!\not\;\; h$ and $a=t-1$.

\begin{fact}\label{fact:pzerovalue}
   Let $h$ and $t$ be positive integers such that $h\geq 2t+2$. Let $p_0=\ell_t^{-1}=\left\lceil\frac{h}{2t+1}\right\rceil^{-1}$ and recall that
   \begin{align*}
      \gamma_{\hh}(p) = \min_{a\in\{0,\ldots,t\}}\left\{\frac{p}{t+1}, \frac{p(1-p)}{a(1-p)+(\ell_a-1)p}\right\} .
   \end{align*}
   Then $\gamma_{\hh}(p)=p/(t+1)$ for $p\in [0,p_0]$.
\end{fact}

We have $\ed_{\hh}(p)\leq\gamma_{\hh}(p)$ and the previous case gives that the two functions are equal at $p=p_0$. They are also equal at $p=0$. By Fact~\ref{fact:pzerovalue}, the function $\gamma_{\hh}(p)$ is linear over $p\in [0,p_0]$ for $h\geq 2t+2$. By Proposition~\ref{prop:basic}, $\ed_{\hh}(p)$ is continuous and concave down, so we may conclude that $\ed_{\hh}(p)=\gamma_{\hh}(p)=\frac{p}{t+1}$ for $p\in [0,p_0]$.~\\

This concludes Case 5 and completes the proof of Theorem~\ref{Thm:main}.
\end{proof}

\section{Proofs of Lemmas and Facts}
\label{sec:proofs}
\begin{proof}[Proof of Corollary \ref{cor:established_gamma}]
The case of $t=1$ is covered by Corollary~\ref{cor:cycles}.

Let $a\in\{1,\ldots,t-1\}$.
\begin{align*}
   &\mbox{If } p \geq \frac{a}{a+\ell_0-\ell_a}, &&\mbox{then } \frac{p(1-p)}{a(1-p)+(\ell_a-1)p} \geq \frac{1-p}{\ell_0-1} . \\
   &\mbox{If } p \leq \frac{t-a}{t-a+\ell_a-\ell_t}, &&\mbox{then }
\frac{p(1-p)}{a(1-p)+(\ell_a-1)p} \geq \frac{p(1-p)}{t(1-p)+(\ell_t-1)p}. \\
\end{align*}

Therefore, it suffices to show
\begin{align}
   \frac{t-a}{t-a+\ell_a-\ell_t} &\geq \frac{a}{a+\ell_0-\ell_a} \nonumber \\
   (\ell_0-\ell_a)(t-a) &\geq (\ell_a-\ell_t)a . \label{eq:suffices}
\end{align}

To that end,
\begin{align*}
   (\ell_0-\ell_a)(t-a)-(\ell_a-\ell_t)a &= (t-a)\ell_0+a\ell_t-t\ell_a \\
   &> \frac{(t-a)h}{t+1}+\frac{ah}{2t+1}-\frac{th}{t+a+1}-t \\
   &= \frac{at(t-a)h}{(t+1)(t+a+1)(2t+1)}-t \\
   &\geq \frac{t(t-1)h}{(t+1)(2t)(2t+1)}-t .
\end{align*}

If $h\geq 4t^2+10t+12+\frac{12}{t-1}$, then (\ref{eq:suffices}) is satisfied and the corollary follows.
\end{proof}

\begin{proof}[Proof of Fact~\ref{fact:two_part_simple}]
We only need to prove one direction because $x$ and $y$ are arbitrary.  In both cases, we will prove the forward implication.
\begin{itemize}
\item[\textit{(\ref{case:xygeq})}] Let $\lfloor h/x\rfloor\geq y$ and $h=qx+r$, where $r\in\{0,\ldots,x-1\}$. Then $y\leq\lfloor h/x\rfloor=q$, so $h\geq xy+r$. Thus $\lfloor h/y\rfloor\geq x+\lfloor r/y\rfloor \geq x$.
\item[\textit{(\ref{case:xyleq})}] Let $\lceil h/x\rceil \leq y$ and $h=qx-r$, where $r\in\{0,\ldots,x-1\}$. Then $y\geq\lceil h/x\rceil=q$, so $h\leq yx-r$. Thus $\lceil h/y\rceil\leq x-\lfloor r/y\rfloor\leq x$.
\end{itemize}
\end{proof}
\begin{proof}[Proof of Fact \ref{fact:bounds_on_h}] Clearly, if $a\in\{0,\ldots,t-1\}$, then $\left\lceil\frac{h}{t+a+1}\right\rceil\leq\left\lceil\frac{h}{t+1}\right\rceil$ so it suffices to prove this fact for $a=0$. Let $h=qt+r$ with $r\in \{0,\ldots,t-1\}$. Since $h\geq t(t-1)$, we have $q\geq t-1 \geq r$. Then 
\begin{align*}
\left\lceil \frac{h}{t+1}\right\rceil = q + \left\lceil \frac{r-q}{t+1}\right\rceil \leq q = \left\lfloor \frac{h}{t}\right\rfloor.
\end{align*} 
\end{proof}

\begin{proof}[Proof of Fact~\ref{fact:tlarge}] If $h\geq (t+1)^2+1$, then $t+2\leq\lceil h/(t+1)\rceil=\ell_0$. Consequently,
\begin{align*}
   t+1\leq\frac{1}{2}(\ell_0+t)\leq p(\ell_0+t)
\end{align*}
and so $\frac{1-p}{\ell_0-1}\leq\frac{p}{t+1}$.

For $a\in\{1,\ldots,t\}$, let $h=q(t+1)+r$, where $r\in\{1,\ldots,t+1\}$. The bound $h\geq (t+1)(t+a)+1$ ensures $q\geq t+a$. Then,
\begin{align*}
   a+\left\lceil\frac{h}{t+a+1}\right\rceil &= a+\left\lceil\frac{q(t+a+1)+r-qa}{t+a+1}\right\rceil \\
   &= q+\left\lceil\frac{a(t+a+1)+r-qa}{t+a+1}\right\rceil \\
   &\leq q+\left\lceil\frac{a(t+a+1)+t+1-(t+a)a}{t+a+1}\right\rceil \\
   &\leq q+1=\left\lceil\frac{h}{t+1}\right\rceil 
\end{align*}
and so $\frac{1-p}{\ell_0-1}\leq\frac{p(1-p)}{a(1-p)+(\ell_a-1)p}$.
\end{proof}

\begin{proof}[Proof of Lemma~\ref{lem:degrees}]~\\
\begin{itemize}

\item[\textit{(\ref{case:deg})}]
Let $v\in\vk$. Using Proposition~\ref{Prop:M_sym},
\begin{alignat*}{2}
\deg_G(v) & \geq \left\lceil\frac{d_G(v)}{\max \lbrace \x(u)\rbrace}\right\rceil \geq \left\lceil \frac{\frac{p-g_K(p)}{p}+\frac{1-2p}{p}\x(v)}{\frac{g_K(p)}{1-p}}\right\rceil \\
&\geq \frac{(p-g_K(p))(1-p)}{pg_K(p)}=\frac{1-p}{g_K(p)}-\frac{1-p}{p}\\
&> \underset {a' \in \lbrace 0,1,\ldots ,t \rbrace }{\max}\left\lbrace\frac{(a'-a)(1-p)+\left( \ell_{a'} -1 \right) p}{p}-\frac{1-p}{p}\right\rbrace\\
&=\underset {a' \in \lbrace 0,1,\ldots ,t \rbrace }{\max} \left\lbrace \frac{(a'-a-1)(1-p)}{p} + \ell_{a'} -1 \right\rbrace\\
&\geq \ell_{a+1} -1.
\end{alignat*}
The last inequality is obtained by choosing $a'=a+1$.

\item[\textit{(\ref{case:codeg})}]
By inclusion-exclusion, $1\geq d_G(v) + d_G(w)-d_G(v,w)$, we have that $d_G(v,w)\geq 2\frac{p-g_K(p)}{p}+\frac{1-2p}{p}(\x(v)+\x(w))-1>\frac{p-2g_K(p)}{p}$. Therefore,
\begin{alignat*}{2}
\deg_G(v,w) & \geq \left\lceil\frac{d_G(v,w)}{\max \lbrace \x(u)\rbrace}\right\rceil \geq \left\lceil \frac{\frac{p-2g_K(p)}{p}}{\frac{g_K(p)}{1-p}}\right\rceil =\frac{1-p}{g_K(p)}-\frac{2(1-p)}{p}\\
&> \underset {a' \in \lbrace 0,1,\ldots ,t \rbrace }{\max}\left\lbrace\frac{(a'-a)(1-p)+\left( \ell_{a'} -1 \right) p}{p}-\frac{2(1-p)}{p}\right\rbrace\\
&= \underset {a' \in \lbrace 0,1,\ldots ,t \rbrace }{\max}\left\lbrace\frac{(a'-a-2)(1-p)}{p}+ \ell_{a'} -1 \right\rbrace\\.
\end{alignat*}

If $a\leq t-2$, then we choose $a'=a+2$. Then $\deg_G(v,w)>\ell_{a+2}-1$, and because $\deg_G(v,w)$ is an integer, $\deg_G(v,w)\geq\ell_{a+2}$.

If $a=t-1$, then we choose $a'=t$. Then $\deg_G(v,w)>-\frac{1-p}{p}+\ell_t-1=\ell_t-p^{-1}\geq 0$, since $p\geq p_0=\ell_t^{-1}$. Because $\deg_G(v,w)$ is an integer, $\deg_G(v,w)\geq 1$.
\end{itemize}
\end{proof}

\begin{proof}[Proof of Lemma~\ref{lem:cycle_length_long}]~\\
\indent We say that a long cycle is a cycle of length at least $L+1$ and will show that there are no long cycles. Let $v_1\cdots v_{\ell}$ be a smallest cycle in $G$ among all those length greater than $L$.~\\

\noindent\textbf{Case 1:} $0\leq a\leq t-2$.~\\
\indent Observe that this case requires $t\geq 2$. Consider the path $v_1\cdots v_{\ell_a-1}$ on the cycle $v_1 \cdots v_{\ell} v_1$. There is no cycle of length $\ell_a$ and so the common neighbors of $v_1$ and $v_{\ell_a-1}$ are all in $\{v_2,\ldots,v_{\ell_a-2}\}$. Note that Lemma~\ref{lem:degrees} establishes that $v_1$ and $v_{\ell_a-1}$ have at least $\ell_{a+2}\geq 2$ common neighbors.

Since all common neighbors of $v_1$ and $v_{\ell_a-1}$ are in $\{v_2,\ldots,v_{\ell_a-2}\}$, we have $\ell_a-3\geq\ell_{a+2}$. Hence,
\begin{align*}
   \frac{h}{t+a+3}\leq\left\lceil\frac{h}{t+a+3}\right\rceil\leq\left\lceil\frac{h}{t+a+1}\right\rceil-3<\frac{h}{t+a+1}-2
\end{align*}
and so $h>(t+a+1)(t+a+3)$.

This gives that the number of common neighbors of $v_1$ and $v_{\ell_a-1}$ is at least $\ell_{a+2}=\left\lceil\frac{h}{t+a+3}\right\rceil\geq t+a+2\geq 4$.

Therefore, $v_1$ and $v_{\ell_a-1}$ has at least two common neighbors in $\{v_3,\ldots,v_{\ell_a-3}\}$. Let $i>2$ and $j<\ell_a-2$ be, respectively, the smallest and largest indices of vertices in $\{v_3,\ldots,v_{\ell_a-3}\}$ that are common neighbors of $v_1$ and $v_{\ell_a-1}$. That is, $3\leq i\leq j\leq\ell_a-3$. The cycle $v_1v_iv_{i+1}\cdots v_{\ell-1}v_{\ell}$ has length $\ell-i+2$. The cycle $v_1v_2\cdots v_{j-1}v_jv_{\ell_a-1}v_{\ell_a}\cdots v_{\ell-1}v_{\ell}$ has length $j+\ell-\ell_a+2$.

Since these two cycles have length less than $\ell$, they cannot be long cycles. Hence, their length is at most $\ell_a-1$, giving us
\begin{align*}
   \ell-i+2 &\leq \ell_a-1 \\
   \ell + j-\ell_a+2 &\leq \ell_a-1 .
\end{align*}
We can add these inequalities and use the fact that $\ell\geq L+1$. Rearranging the terms, we conclude the following:
\begin{align}
   3\ell_a-2L-7\geq 3\ell_a-2\ell-5\geq j-i+1
   \geq\ell_{a+2}-2. \label{eq:longcycles}
\end{align}

To verify there are no long cycles, we must show that (\ref{eq:longcycles}) produces a contradiction. Since $0\leq a\leq t-2$,
\begin{align*}
   3\ell_{a}-2L-7 &= 3\left\lceil\frac{h}{t+a+1}\right\rceil -2\left\lfloor\frac{h}{t}\right\rfloor -7 \\
   &< 3\left(\frac{h}{t+a+1}+1\right)-2\left(\frac{h}{t}-1\right)-7 \\
   &= \frac{h}{t+a+3}-2-\frac{2h(at+a^2+4a+3)}{t(t+a+1)(t+a+3)} \\
   &< \left\lceil\frac{h}{t+a+3}\right\rceil-2 = \ell_{a+2}-2 ,
\end{align*}
a contradiction for all $t\geq 2$ and $h\geq 2t+2$. Therefore, for $0\leq a\leq t-2$, $G$ has no cycle of length longer than $\ell_a-1$.~\\

\noindent\textbf{Case 2:} $a=t-1$.~\\
\indent Since all common neighbors of $v_1$ and $v_{\ell_{t-1}-1}$ are in $\{v_2,\ldots,v_{\ell_{t-1}-2}\}$, we have $\ell_{t-1}-3\geq 1$. Hence,
\begin{align*}
   1\leq\ell_{t-1}-3=\left\lceil\frac{h}{2t}\right\rceil-3<\frac{h}{2t}-2
\end{align*}
and so $h>6t$, which means $\ell_{t-1}\geq 4$.

Consider the path $v_1\cdots v_{\ell_{t-1}}$ on the cycle $v_1 \cdots v_{\ell} v_1$. To see there is no cycle of length $\ell_{t-1}+1$, we set $h=q(2t)-r$ with $q\geq 2$ and $r\in\{0,\ldots,2t-1\}$ and have
\begin{align*}
   \ell_{t-1}+1 = \left\lceil\frac{h}{2t}\right\rceil+1 = q+1 \leq 2q-2 \leq 2q+\left\lfloor\frac{-r}{t}\right\rfloor \leq \left\lfloor\frac{h}{t}\right\rfloor = L .
\end{align*}

Since $v_1$ and $v_{\ell_{t-1}}$ have a common neighbor $v_i$, either $v_1v_iv_{i+1}\cdots v_{\ell}v_1$ or  $v_1\cdots v_iv_{\ell_{t-1}}v_{\ell_{t-1}+1}\cdots v_{\ell}v_1$ has length less than $\ell$.  Without loss of generality, we will assume that it is the former. This gives
\begin{align*}
   \ell_{t-1}-1 \geq \ell-i+2 \geq \ell-(\ell_{t-1}-1)+2 \geq L+1-(\ell_{t-1}-1)+2 .
\end{align*}

Consequently,
\begin{align}
   2\ell_{t-1}-L-5 \geq 0 . \label{eq:longcycles:tminus1}
\end{align}
To see that (\ref{eq:longcycles:tminus1}) is contradicted,
\begin{align*}
   2\ell_{t-1}-L-5 &= 2\left\lceil\frac{h}{2t}\right\rceil -\left\lfloor\frac{h}{t}\right\rfloor -5 \\
   &< 2\left(\frac{h}{2t}+1\right)-\left(\frac{h}{t}-1\right)-5 = -2< 0 .
\end{align*}

Therefore, for $a=t-1$, $G$ has no cycle of length longer than $\ell_{t-1}-1$.~\\
\end{proof}

\begin{proof}[Proof of Proposition~\ref{prop:Hamiltonian}]
Let $v_1\cdots v_{\ell}$ be a longest path in $G$ such that $v_1v_{\ell}\in E(G)$. If $G$ is not Hamiltonian, there exists a $w\in V(G)-\{v_1,\ldots,v_{\ell}\}$. Because $G$ is connected, there exists $i\in\{1,\ldots,\ell\}$ and $w'\in V(G)-\{v_1,\ldots,v_{\ell}\}$ such that $v_i$ is adjacent to $w'$. Then there is a longer path: $v_{i+1}\cdots v_{\ell}v_1\cdots v_iw'$, a contradiction.
\end{proof}

\begin{proof}[Proof of Lemma~\ref{Lem:gap}]~\\
Because $v_1 \cdots v_{\ell}$ is a longest path in $F$, neither $v_1$ nor $v_k$ can have neighbors off this path, as that would yield a longer path. Thus $N(v_1)\cup N(v_{\ell})\subseteq\{v_1,\ldots,v_k\}$ in $F$.~\\

\noindent\textbf{Case 1:} $\ell\leq\ell_a$.~\\
\indent If $v_i$ is adjacent to $v_1$, then $v_{i-1}$ cannot be adjacent to $v_{\ell}$. Thus, the predecessors of $N(v_1)$ and the neighbors of $v_{\ell}$ are disjoint subsets in $\{v_1,\ldots,v_{\ell-1}\}$. Since both $v_1$ and $v_{\ell}$ have degree at least $\ell_{a+1}$, hence
\begin{align*}
   2\ell_{a+1}\leq \ell-1\leq \ell_a-1 .
\end{align*}
However,
\begin{align}
   \ell_a-2\ell_{a+1}-1 &= \left\lceil\frac{h}{t+a+1}\right\rceil-2\left\lceil\frac{h}{t+a+2}\right\rceil-1 \nonumber \\
   &< \frac{h}{t+a+1}-\frac{2h}{t+a+2}=-\frac{h(t+a)}{(t+a+1)(t+a+2)}<0 . \label{eq:lala}
\end{align}~\\

\noindent\textbf{Case 2:} $\ell\geq\ell_a+1$.~\\
\indent Partition the vertices of this path into $2s+1$ consecutive sets $A_0,B_1,A_1,\ldots,A_s,B_s$ with $s\geq 0$, constructed so that, in each set $A_i$, neighbors of $v_1$ appear before neighbors of $v_{\ell}$ as follows:

We let neighbors of $v_1$ be denoted with $v_{p_i}$ and neighbors of $v_{\ell}$ be denoted with $v_{q_i}$ in this construction. Let $A_0$ contain $v_1$ and add consecutive vertices of this path until we arrive at a neighbor of $v_{\ell}$. From this point forward we do not allow another neighbor of $v_1$ to be in $A_0$, i.e. we continue adding consecutive vertices until we reach the last neighbor $v_{q_0}$ of $v_{\ell}$ before another neighbor $v_{p_1}$ of $v_1$. Then $A_0=\{v_1,\ldots,v_{q_0}\}$, and we define $B_1=\{v_{q_0+1},\ldots,v_{p_1-1}\}$. Note that this definition does not preclude $B_1$ being an empty set. Continuing with this algorithm, we define sets $A_1=\{v_{p_1},\ldots,v_{q_1}\}$ and $B_2=\{v_{q_1+1},\ldots,v_{p_2-1}\}$, where $v_{p_1}$ is a neighbor of $v_1$ on this path, $v_{q_1}$ is the last neighbor of $v_{\ell}$ in $A_1$ before another neighbor $v_{p_2}$ of $v_1$ as shown in Figure~\ref{Figure}. We continue in this way and define sets $A_i=\{v_{p_i},\ldots,v_{q_i}\}$ and $B_i=\{v_{q_{i-1}+1},\ldots,v_{p_i-1}\}$ for $i\in\{1,\ldots,s\}$, adding the last vertex $v_{\ell}$ into the set $A_s$.

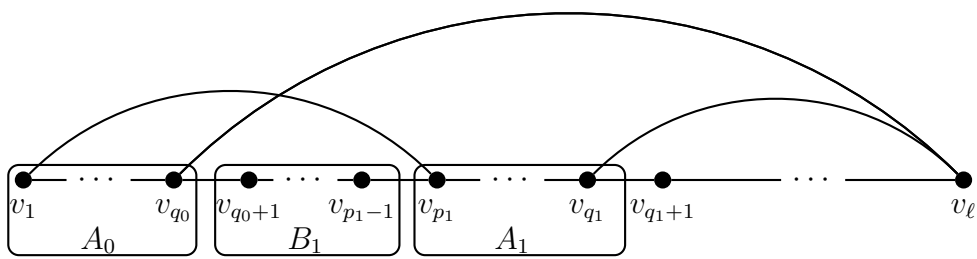
\begin{figure}[!h]
\centering
\begin{tikzpicture}[bend angle=45]
\tikzstyle{knode}=[circle,draw=black,fill=black,thick,inner sep=2pt]
\tikzstyle{place}=[circle,draw=blue!50,fill=blue!20,thick,inner sep=0pt,minimum size=6mm]

\node (v_1) at (0,0) [knode,label=below:$v_1$] {};
\node (v_2_dots) at (1,0)  [] {$\cdots$};
\node (v_q_0) at (2,0) [knode,label=below:$v_{q_0}$] {};
\node (v_q_0+1) at (3,0) [knode,label=below:$v_{q_0+1}$] {};
\node (v_q_0+1_dots) at (3.75,0) [] {$\cdots$};
\node (v_p_1-1) at (4.5,0) [knode,label=below:$v_{p_1-1}$] {};
\node (v_p_1) at (5.5,0) [knode,label=below:$v_{p_1}$] {};
\node (v_p_1_dots) at (6.5,0) [] {$\cdots$};
\node (v_q_1) at (7.5,0) [knode,label=below:$v_{q_1}$] {};
\node (v_q_1+1) at (8.5,0) [knode,label=below:$v_{q_1+1}$] {};
\node (v_q_1+1_dots) at (10.5,0) [] {$\cdots$};
\node (v_ell) at (12.5,0) [knode,label=below:$v_{\ell}$] {};
\path (v_1) edge [thick] node {} (v_2_dots)
(v_1) edge [bend left,thick] node {} (v_p_1)
(v_q_0) edge [bend left, thick] node {} (v_ell)
(v_q_0) edge [thick] node {} (v_q_0+1)
(v_p_1-1) edge [thick] node {} (v_p_1)
(v_q_1) edge [thick] node {} (v_q_1+1)
(v_q_0) edge [bend left, thick] node {} (v_ell)
(v_q_1) edge [bend left, thick] node {} (v_ell)
(v_2_dots) edge [thick] node {} (v_q_0)
(v_q_0+1) edge [thick] node {} (v_q_0+1_dots)
(v_q_0+1_dots) edge [thick] node {} (v_p_1-1)
(v_p_1) edge [thick] node {} (v_p_1_dots)
(v_p_1_dots) edge [thick] node {} (v_q_1)
(v_q_1+1) edge [thick] node {} (v_q_1+1_dots)
(v_q_1+1_dots) edge [thick] node {} (v_ell);
\draw[rounded corners,thick] (-0.2,-1) rectangle (2.3,0.2);
\node (A_0) at (1,-0.8) {$A_0$};
\draw[rounded corners,thick] (2.55,-1) rectangle (5,0.2);
\node (B_1) at (3.75,-0.8) {$B_1$};
\draw[rounded corners,thick] (5.2,-1) rectangle (8,0.2);
\node (A_1) at (6.5,-0.8) {$A_1$};

\end{tikzpicture}
\caption{Partition of vertices of the path. Sets $A_i$ are iteratively constructed so that they contain consecutive vertices of this path starting with a neighbor of $v_1$ and ending with the last neighbor of $v_{\ell}$ so that no neighbor of $v_1$ appears after neighbors of $v_{\ell}$ in each set. Sets $B_i$ contain consecutive vertices between sets $A_{i-1}$ and $A_i$, if there are any. The first vertex is placed in $A_0$ and the last vertex $v_{\ell}$ in $A_s$.}\label{Figure}\end{figure}

Now we analyze this partition:
\begin{itemize}
   \item We call the sets $B_i$, $i\in\{1,\ldots,s\}$, \textit{gaps} as they do not contain any neighbors of either $v_1$ or $v_{\ell}$, but only contain vertices that succeed a given neighbor of $v_{\ell}$ and precede a given neighbor of $v_1$. According to the definition, gaps may be empty, but we will see below that this is not possible in this case.
   \item Each set $A_i$, $i\in\{0,\ldots,s\}$, contains at most one common neighbor of $v_1$ and $v_{\ell}$.
   \item By construction, neighbors of $v_1$ (other than a common neighbor, if exists) precede neighbors of $v_{\ell}$ in each $A_i$, $i\in\{0,\ldots,s\}$.
\end{itemize}

It will suffice to show that $s=0$. This will imply that no neighbor of $v_1$ follows the first neighbor of $v_{\ell}$ on this path, which further implies that $N(v_1)$ entirely precedes $N(v_{\ell})$, except possibly for a single common vertex. Since $v_1$ and $v_{\ell}$ have at least one common neighbor, the lemma will follow.

Notice that $v_1 \cdots v_{q_0} v_{\ell} v_{\ell-1} \cdots v_{p_1} v_1$ is a cycle as seen in Figure~\ref{Figure}. In fact, for any $i\geq 1$, removing the gap $B_i$ from vertices $\{v_1,\ldots,v_{\ell}\}$ forms a cycle, so by assumption, $\ell-|B_i|\leq\ell_a-1$ and none of the gaps can be empty. Therefore, $\sum_{i=1}^s |B_i|\geq s(\ell-\ell_a+1)$.

On the other hand, by the degree assumption and since each set $A_i$ contains at most one common neighbor of $v_1$ and $v_{\ell}$, we obtain $2\ell_{a+1}\leq |N(v_1)|+|N(v_{\ell})|\leq \left(\sum_{i=0}^s|A_i|\right)+(s+1)-2$. Combining these two inequalities we have
\begin{align*}
   \ell = \sum_{i=0}^s|A_i| + \sum_{i=1}^s|B_i| &\geq 2\ell_{a+1} - (s+1) + 2 + s(\ell-\ell_a+1) \\
   &= s(\ell - \ell_a) + 2\ell_{a+1} + 1.
\end{align*}

If $s\geq 1$, then we have $\ell\geq\ell-\ell_a+2\ell_{a+1}+1$ which simplifies to $\ell_a-2\ell_{a+1}-1\geq 0$, which is contradicted by (\ref{eq:lala}). Therefore $s=0$ and the lemma follows.
\end{proof}

\begin{proof}[Proof of Fact~\ref{fact:pzerovalue}]
We need to show that $\gamma_{\hh}(p_0)=p_0/(t+1)$. Since
\begin{align*}
   \gamma_{\hh}(p_0)=p_0\cdot\min_{a\in\{0,\ldots,t\}} \left\{\frac{1}{t+1},\frac{1-p_0}{a(1-p_0)+\left(\ell_a-1\right)p_0}\right\} ,
\end{align*}
we need to show that $\frac{\ell_a-1}{\ell_t-1}\leq t-a+1$ for all $a\in\{0,\ldots,t-1\}$.

To do this, let $h=q(2t+1)-r$ where $r\in\{0,\ldots,2t\}$ and $q\geq 2$ (because $h\geq 2t+2$). Then,
\begin{align*}
   \frac{\ell_a-1}{\ell_t-1} &= \frac{1}{q-1}\left(q-1+\left\lceil\frac{q(t-a)-r}{t+a+1}\right\rceil\right) \\
   &\leq \frac{1}{q-1}\left(q-1+\frac{q(t-a)+t+a}{t+a+1}\right) \\
   &= t-a+1+\frac{t^2-a^2+2t-q(t^2-a^2)}{(q-1)(t+a+1)} ,
\end{align*}
which is at most $t-a+1$ if $q\geq 3$ or if $a\leq t-2$ and $q=2$.  In the case where $a=t-1$ and $q=2$, then $\frac{\ell_a-1}{\ell_t-1}=1+\left\lceil\frac{2-r}{2t}\right\rceil\leq 2=t-a+1$.
\end{proof}

\section{Conclusion and open questions}
\label{sec:divisibility_p_small}

We have obtained the edit distance function over all of its domain for $C_h^t$ when $t+1$ does not divide $h$ and $h\geq 2t(t+1)+1$. When $t+1$ divides $h$ and $h\geq 2t(t+1)+1$, we have obtained the function for $p\in [p_0,1]$, where $p_0=\left\lceil\frac{h}{2t+1}\right\rceil^{-1}$. The function, however, is not known when $t+1$ divides $h$ and $p\in [0,p_0)$ or when $h\leq 2t(t+1)$. 

As to the case of $p<p_0$ (and $h$ sufficiently large), we showed that if $K\in\K(\forb(C_h^t))$ is a $p$-core CRG with $p<1/2$ which has $a\neq t-1$ white vertices, then $g_K(p)= \gamma_{\forb(C_h^t)}(p)$. Therefore, to solve the problem for the remaining case when $t+1$ divides $h$, and $p$ is small, one only needs to consider CRGs with exactly $t-1$ white vertices. A particular barrier to this is Lemma~\ref{lem:degrees} which requires $p\geq p_0$ to ensure that the graph induced by the black vertices and gray edges of the CRG has the property that any two vertices have at least one common neighbor. Such a condition need not hold for small $p$.

As to reducing the lower bound required of $h$, we note that in the proof of Theorem~\ref{Thm:main}, we required $h\geq 2t(t+1)+1$ in Fact~\ref{fact:tlarge}.  This ensured that the $\gamma_{\hh}$ function for $p\in [1/2,1]$ was linear and by the concavity and continuity of the edit distance function (see Proposition~\ref{prop:basic}), this ensures that $\ed_{\hh}(p)=\gamma_{\hh}(p)$ in that interval.  So, more careful analysis of the case $p\geq 1/2$ may enable one to reduce the lower bound on $h$, but these arguments are very different from the case where $p<1/2$. Elsewhere, we only require $h\geq\max\{t(t-1),2t+2\}$ in order to complete the proof of Theorem~\ref{Thm:main}. This bound is required in several places. See Fact~\ref{fact:bounds_on_h}, Lemma~\ref{lem:cycle_length_short}, Lemma~\ref{lem:cycle_length_long} but especially the basic Fact~\ref{fact:numth} which says that a set of size $h$ can be partitioned into sets of size $t$ or $t+1$ if and only if $h\geq t(t-1)$. So we believe that it would be difficult to prove the theorem for values of $h$ smaller than $\max\{t(t-1),2t+2\}$ in general.

\section*{Acknowledgements}
\label{sec:acknowl}

Berikkyzy was supported through Wolfe Research Fellowship of Department of Mathematics at Iowa State University. Martin's research was partially supported by the National Security Agency (NSA) via grant H98230-13-1-0226. Martin's contribution was completed in part while he was a long-term visitor at the Institute for Mathematics and its Applications. He is grateful to the IMA for its support and for fostering such a vibrant research community. Peck's research is based upon work supported by the National Science Foundation Graduate Research Fellowship under Grant No. DGE0751279.

\end{document}